\documentclass[a4paper,10pt]{amsart}
\usepackage[latin1]{inputenc}
\usepackage[T1]{fontenc}
\usepackage{amsmath}
\usepackage{amsthm}
\usepackage{latexsym}
\usepackage{amssymb}
\usepackage{color}
\usepackage[all]{xy}
\usepackage{calc}
\usepackage{multicol} 
\usepackage[left=3cm,right=3cm]{geometry}
\usepackage{enumitem}
\usepackage{hyperref}  

\newtheorem{thm}{Theorem}[section]
\newtheorem{prop}{Proposition}[section]

\newtheorem{lem}{Lemma}[section]
\newtheorem{cor}{Corollary}[section]

\newtheorem{rmq}{Remark}[section]

\newcommand{\R}{\mathbb{R}}
\numberwithin{equation}{section}
\newcommand{\N}{\mathbb{N}}
\newcommand{\Z}{\mathbb{Z}}
\newcommand{\Cc}{\mathbb{C}}

\newcommand{\sm}{\setminus}

\newcommand{\eps}{\varepsilon}

\newcounter{exercice}

\DeclareMathOperator{\Real}{Re}
\DeclareMathOperator{\supp}{supp}

\renewcommand{\b}[1]{\textcolor{blue}{#1}}

\begin{document}
\title[4th order NHL]{On a fourth order nonlinear Helmholtz equation}

\author[Bonheure, Casteras and Mandel]{Denis Bonheure \and Jean-Baptiste Casteras \and Rainer Mandel}

\address{Denis Bonheure, Jean-Baptiste Casteras
\newline \indent D\'epartement de Math\'ematiques, Universit\'e Libre de Bruxelles,
\newline \indent CP 214, Boulevard du triomphe, B-1050 Bruxelles, Belgium,
\newline \indent and INRIA- team MEPHYSTO.}
\email{Denis.Bonheure@ulb.ac.be}
\email{jeanbaptiste.casteras@gmail.com}
\address{Rainer Mandel
\newline \indent Karlsruhe Institute of Technology
\newline \indent  Institute for Analysis
\newline \indent  Englerstrasse 2, D-76131 Karlsruhe, Germany.}
\email{Rainer.Mandel@kit.edu }

\begin{abstract}
In this paper, we study the mixed dispersion fourth order nonlinear Helmholtz equation
$$
 \Delta ^2 u -\beta \Delta u + \alpha u= \Gamma|u|^{p-2} u \quad\text{in } \R^N,
$$
for positive, bounded and $\Z^N$-periodic functions $\Gamma$ in the following three cases:
\begin{equation*}
  (a)\;\; \alpha<0,\beta \in \R \qquad\text{or}\qquad
  (b)\;\; \alpha>0,\beta < -2\sqrt{\alpha} \qquad\text{or}\qquad
  (c)\;\; \alpha =0,\beta <0.
\end{equation*}
Using the dual method of Ev\'equoz and Weth, we find solutions to this equation and establish some of their
qualitative properties.
\end{abstract}

\maketitle
\section{Introduction}
In this paper, we study the existence and the qualitative properties of solutions to the following mixed
dispersion fourth order nonlinear Helmholtz type equation
\begin{equation}
\label{4nle}
\tag{4NHE}
\gamma \Delta^2 w - \Delta w  +\alpha w = \Gamma |w|^{p-2}w \quad \text{in } \R^N,
\end{equation}
where $\gamma>0$, $\alpha \in \R,p>2$ and $\Gamma$ is a positive, bounded and periodic function. When
$\gamma=0$, \eqref{4nle} yields standing wave solutions, i.e. solutions of the form $\psi
(t,x)=e^{i\alpha t}w(x)$, to the well-known Schr\" odinger equation
\begin{equation}
\label{2NLS}
\tag{2NLS}
i\partial_t \psi +\Delta \psi + |\psi|^{p-2}\psi =0,\  \psi (0,x)=\psi_0 (x),\quad (t,x)\in \R\times \R^N.
\end{equation}
It is well-known that when $(p-2)N< 4$, solutions to \eqref{2NLS} exist globally in time
and that they are stable whereas when $(p-2)N\geq 4$, they can become singular in finite time and they are unstable \cite{Caz,Sulem}. 
Observe that in the physically relevant case $N=2$ and $p=4$,
we are in the second situation. In order to regularize and stabilize solutions to \eqref{2NLS}, Karpman and
Shagalov \cite{MR1779828} introduced a small fourth-order dispersion, namely they considered
\begin{equation}
\label{4NLS}
\tag{4NLS}
i\partial_t \psi -\gamma \Delta^2 \psi +\Delta \psi +|\psi|^{p-2}\psi =0,\ \psi (0,x)=\psi_0 (x),\quad
(t,x)\in \R\times \R^N.
\end{equation}
Using a combination of stability analysis and numerical simulations, they showed  that standing wave solutions
for this equation, i.e. solutions of \eqref{4nle} with $\alpha>0$, are stable when $(p-2)N<8$
and unstable when $(p-2)N\geq 8$. Fibich and al. \cite{MR1898529} also  proved, using the Strichartz estimates
of Ben-Artzi and al. \cite{MR1745182}, global existence in time of the solutions to \eqref{4NLS} when $(p-2)N<8$ when the initial datum is in the energy space. In particular, thanks to the presence of the biharmonic term, we see that when $N=2$ and $p=4$, solutions to \eqref{4NLS} exist globally in time and standing solutions are stable. The addition of the fourth order term has also been motivated from a phenomenological point of view. In nonlinear optics, \eqref{2NLS} is usually derived from the nonlinear Helmholtz equation through the so-called paraxial approximation. The fact that its solutions may blow up in finite time suggests that some small terms neglected in the paraxial approximation play an important role to prevent this phenomenon. The addition of a small fourth-order dispersion term was proposed in \cite{MR1898529} as a nonparaxial
correction, which eventually gives rise to \eqref{4NLS}. Despite being less studied than the classical \eqref{2NLS}, an increasing attention has been given to
 \eqref{4NLS}. We refer to the works of Pausader \cite{MR2353631,MR2502523,MR2505703,MR2746203,MR3078112},
 Miao and al. \cite{miao}, Ruzhansky and al. \cite{ruz}, Segata \cite{segata1,segata2} concerning global
 well-posedness and scattering, to \cite{BCGJ2,BL} for finite-time blow-up and to
 \cite{BCDN,BCGJ1,natalipastor} for the stability of standing wave solutions.
We also mention that \eqref{4nle} also appears in the theory of water waves \cite{bretherton} and as a model to study travelling waves in suspension bridges \cite{lazer,mckenna} (see also \cite{Buffoni1995109,levandosky}). \\

Next, let us mention existence results for \eqref{4nle}. First, observe that using the scaling $u(x)=w(\gamma^{-1/4}x)$, we see that \eqref{4nle} is equivalent to
\begin{equation}
\label{4nls}
 \Delta^2 u - \beta\Delta u  +\alpha u = |u|^{p-2}u \quad\text{in } \R^N,
\end{equation}
where $\beta =\gamma^{-1/2}$. Bonheure and Nascimento \cite{BN} considered the following minimization problem
\begin{equation}
\label{mini}
m:=\inf_{u\in M} J_{\alpha,\beta}(u), 
\end{equation}
where
$$
  J_{\alpha,\beta}(u):=\int_{\R^N} (|\Delta u|^2 +\beta |\nabla u|^2 +\alpha u^2) \, dx
$$
and
$$
  M:=\{u\in H^2 (\R^N ):\ \int_{\R^N}|u|^{p} \, dx=1 \}. 
$$
Notice that if $u\in M$ achieves the infimum $m$, then $u$ is a solution to
$$\Delta^2 u -\beta \Delta u+\alpha u=m|u|^{p-2}u. 
$$
Thus, if $m>0$, then $w=m^{1/(p-2)}u$ solves \eqref{4nls}. They showed that the minimization problem
\eqref{mini} admits a solution provided that $\alpha>0$ and
$\beta>-2\sqrt{\alpha}$. The exponent is assumed to satisfy $p>2$ in the case $N=2,3,4$ and $2<p<2N/(N-4)$
if $N>4$. Observe that, thanks to these assumptions on $\alpha$ and $\beta$, the
functional $J_{\alpha,\beta}(u)$ is equivalent to the usual norm in $H^2 (\R^N)$. They also show that their solution has a sign, is radially symmetric if in addition $\beta\geq 2\sqrt{\alpha}$ and, in \cite{BCDN}, that it is exponentially decreasing (if $\beta > -2\sqrt{\alpha}$). 
The case
$\alpha=0$ and $\beta>0$ has been considered in \cite{BCGJ1} where the existence of solutions, belonging to
$X:=\{u\in D^{1,2}(\R^N)|\ \|\Delta u\|_{L^2 (\R^N)}<\infty \}$, has been obtained provided that 
$2N/(N-2)\leq p$  if $N=3,4$ and $2N/(N-2)\leq p< 2N/(N-4)$  if $N>4$. Moreover, this solution has a sign and belongs
to $L^2 (\R^N)$ if and only if $N\geq 5$ suggesting that it decays only polynomially. In fact, it was proved
in this setting that any radial solution $u$ to \eqref{4nls} satisfies $\lim_{|x|\rightarrow
\infty} u(x)|x|^{N-2}=C$ for some constant $C\in \R \backslash \{0\}$. The non radial case is open. 

\medskip 
  
In this paper, we are interested in other ranges for the parameters $\alpha$ and $\beta$, namely we consider the cases
\begin{equation}
  (a)\;\; \alpha<0,\beta \in \R \qquad\text{or}\qquad
  (b)\;\; \alpha>0,\beta < -2\sqrt{\alpha} \qquad\text{or}\qquad
  (c)\;\; \alpha =0,\beta <0.
\end{equation}

To our knowledge, existence results in these cases have not been previously treated in the literature. The main difficulty for these parameter values comes from the fact that $0$ is contained in the essential spectrum of the differential operator
$L:=\Delta^2-\beta\Delta+\alpha$.
Recently, a series of papers by Ev\'equoz and Weth \cite{Ev,MR3149060,EW,MR3625081} tackles this problem
for \eqref{4nle} in the case $\gamma=0,\beta=1$ and $\alpha <0$. We also refer to the previous work of Guti\' errez
\cite{Gu} and the very recent works \cite{Mandel2,Mandel1} of the third author and his collaborators for several extensions respectively to the case where $\alpha$ is replaced by a $\Z^N$-periodic potential and to the case of a system . In \cite{MMP} a sharp decay result for radial solutions was found for the second order Helmholtz equation. 

\smallbreak
We now describe in more details the Ev\'equoz and Weth strategy, see \cite{EW}, that we will adapt. In that paper, the authors
studied the following equation
\begin{equation}
\label{eqEW}
-\Delta u - k^2 u = \Gamma |u|^{p-2}u \quad\text{ in }\R^N,
\end{equation}
where $\Gamma \in L^\infty (\R^N)$, $0\neq \Gamma\geq 0$ and either $\Gamma$ is $\Z^N$-periodic or
$\lim_{|x|\rightarrow \infty}\Gamma (x)=0$. The main difficulty of this problem is the lack of a direct
variational approach. Indeed, one expects that the solutions to \eqref{eqEW} will not decay faster than
$O(|x|^{(1-N)/2})$ as $|x|\rightarrow \infty$. This last claim was indeed proved in~\cite{MMP} for all
nontrivial radial solutions of a class of nonlinear Helmholtz equations of the form \eqref{eqEW}. As a consequence, in this case, the usual energy functional formally associated with \eqref{eqEW} is not even well-defined on
nontrivial solutions.

To overcome this difficulty, Ev\'equoz and Weth proposed a dual variational approach, transforming
\eqref{eqEW} into 
\begin{equation} \label{eq:dual_equation_2ndorder}
  |v|^{p^\prime-2}v= \Gamma^{1/p} \mathcal R_{k^2}\big(\Gamma^{1/p}v \big),
\end{equation} 
where $v=\Gamma^{1/p^\prime}|u|^{p-2} u$ and $\mathcal R_{k^2}=(-\Delta - k^2)^{-1}$ is a resolvent-type
operator constructed via the so-called limiting absorption principle. We refer
to~\eqref{eq:definition_2ndorder_resolvent} below for a precise definition. Here and in the following, $p^\prime=p/(p-1)$.
Thanks to this dual formulation, which is variational in $L^{p^\prime}(\R^N)$, they obtained a ground-state
solution $v\in L^{p^\prime}(\R^N)$ of \eqref{eq:dual_equation_2ndorder} via the Moutain-Pass theorem whenever
the exponent satisfies $\frac{2(N+1)}{N-1}<p<\frac{2N}{N-2}$ and $\Gamma$ is positive, bounded and
$\Z^N$-periodic.
The associated function $u$ was shown to be a strong solution of~\eqref{eqEW} lying in $W^{2,q}(\R^N)\cap
C^{1,\alpha}(\R^N)$ for all $q\in [p,\infty)$ and $\alpha \in (0,1)$. 
Notice that existence results have been obtained also in the Sobolev-critical case $p=\frac{2N}{N-2}$ in
\cite{EvYe}. The lower bound for $p$ is related to the mapping properties of the
resolvent type operator $\mathcal R_{k^2}$,  which in turn is linked with the Stein-Tomas Theorem (see Theorem
\ref{STT}). We will comment on this in more detail later on.

\medskip

Let us return to the nonlinear fourth order Helmholtz equation~\eqref{4nle} that we will investigate by 
adapting the dual variational method of Ev\'equoz and Weth. The main task is to construct and analyze
a resolvent-type operator $\mathfrak R:= (\Delta^2 -\beta \Delta +\alpha)^{-1}$ with mapping properties
similar to and even better than their second order counterparts. First notice that we can
decompose the operator $L$ into two second order operators by writing
\begin{align} \label{eq:def_a1a2}
  \begin{aligned}
  L &= \Delta^2 -\beta \Delta +\alpha = (-\Delta  -a_1 ) (-\Delta -a_2), \quad\text{where } \\
  a_1 &:= \dfrac{-\beta+\sqrt{\beta^2 - 4\alpha}}{2},\qquad a_2:= \dfrac{-\beta-\sqrt{\beta^2 - 4\alpha}}{2}.
  \end{aligned}
\end{align}
We see that $\alpha<0$ implies $a_1>0>a_2$ and $\alpha=0,\beta<0$ implies $a_1>0=a_2$ so that $L$ becomes a
composition of a Schr\" odinger operator and a Helmholtz operator or the Laplacian, respectively. In the case
$\alpha>0$ and $\beta < -2\sqrt{\alpha}$ we find $a_1>a_2>0$ so that $L$ decomposes into two Helmholtz
operators. This leads us to study the nonlinear problem~\eqref{4nls} under the following assumptions:
\begin{itemize}
  \item[(A1)] $\alpha,\beta\in\R$ satisfy $\alpha<0$, $N\geq 2$ or $\alpha>0,\beta< -2\sqrt{\alpha}$, $N\geq
  2$   or $\alpha=0$, $\beta<0$, $N\geq 3$;
  \item[(A2)] $\Gamma\in L^\infty(\R^N)$ is $\Z^N$-periodic with $\inf_{\R^n} \Gamma>0$ and
  $\frac{2(N+1)}{N-1}<p<\frac{2N}{(N-4)_+}$.
\end{itemize}
Here and in the following, the symbol $\frac{2N}{(N-4)_+}$ stands for $\infty$ in the case $N\leq 4$
and for $\frac{2N}{N-4}$ if $N\geq 5$. Besides the mere existence of a nontrivial $L^p(\R^N)$-solution
of~\eqref{4nls}, we will determine further regularity properties as well as a far field pattern for such
solutions. This pattern can be expressed in terms of the function
$$
  U_f(x) := \frac{a_1^{\frac{N-3}{4}}}{\sqrt{\beta^2-4\alpha}} \sqrt{\frac{\pi}{2}} \frac{e^{i(\sqrt{a_1}
  |x|-\frac{N-3}{4}\pi)}}{|x|^{\frac{N-1}{2}}} \hat{f}\big(\sqrt{a_1} \frac{x}{|x|}\big),
  \qquad\text{if }\alpha<0 \text{ or } N>3,\ \alpha=0,\ \beta<0, 
$$
$$ U_f(x) := \frac{1}{|\beta|} \sqrt{\frac{\pi}{2}} \frac{e^{i\sqrt{a_1}
  |x|}}{|x|} \hat{f}\big(\sqrt{a_1} \frac{x}{|x|}\big)- \dfrac{1}{4\pi |\beta||x|}\int_{\R^N} f(x) \, dx ,  
  \qquad\text{if } N=3,\ \alpha=0,\ \beta<0, 
$$
or
\begin{align*}
U_f(x) 
   &:= \frac{a_1^{\frac{N-3}{4}}}{\sqrt{\beta^2-4\alpha}} \sqrt{\frac{\pi}{2}} \frac{e^{i(\sqrt{a_1}
  |x|-\frac{N-3}{4}\pi)}}{|x|^{\frac{N-1}{2}}} \hat{f}\big(\sqrt{a_1} \frac{x}{|x|}\big) \\
  &- \frac{a_2^{\frac{N-3}{4}}}{\sqrt{\beta^2-4\alpha}} \sqrt{\frac{\pi}{2}} \frac{e^{i(\sqrt{a_2}
  |x|-\frac{N-3}{4}\pi)}}{|x|^{\frac{N-1}{2}}} \hat{f}\big(\sqrt{a_2} \frac{x}{|x|}\big),
  \qquad\text{if } \alpha>0,\ \beta < -2\sqrt{\alpha}. 
\end{align*}
Here, $a_1,a_2$ are given as in~\eqref{eq:def_a1a2} and $f$ is chosen such that its Fourier transform on
spheres is well-defined. Our main result is the following.

\begin{thm} \label{mainthm}
  Assume (A1),(A2). Then there exists a nontrivial solution $u\in W^{4,q} (\R^N)\cap
  C^{3,\alpha}(\R^N)$ for all $q\in [p,\infty)$, $\alpha \in (0,1)$ to
  \begin{equation} \label{intro}
  \Delta^2 u - \beta \Delta u +\alpha u = \Gamma |u|^{p-2} u\quad\text{ in } \R^N
  \end{equation}
  satisfying the farfield expansion  
  \begin{equation} \label{eq:farfield_expansion}
    \lim_{R\rightarrow \infty} \frac{1}{R}\int_{B_R} |u(x) - \Real(U_f)(x)|^2 \,dx =0 
  \end{equation}
  for $f:= \Gamma |u|^{p-2}u$.
\end{thm}

As in \cite{EW,MMP} one may put slightly different assumptions on $\Gamma$ that still
ensure the existence of nontrivial solutions. For instance, replacing the periodicity assumption on $\Gamma$
by $\Gamma(x)\to 0$ as $|x|\to\infty$ as in Theorem~1.2 \cite{EW}, the dual variational approach benefits from
even better compactness properties that allow to prove the existence of infinitely many solutions via the
Symmetric Mountain Pass Theorem. Similarly, $\Gamma$ may be replaced by $-\Gamma$ as was pointed out in 
Section~3 of~\cite{MMP}. In Remark~\ref{rmqradial} we also comment on the radially symmetric case where one
can prove the existence of solutions for a strictly larger range of exponents.  
Concerning the qualitative properties of the solution granted by the above theorem, we
can actually say more. Under mild additional assumptions (which are not even needed in the physically most
important case $N=3$) we can show that $u$ satisfies a radiation condition at infinity. Moreover, we will
show that the farfield expansion~\eqref{eq:farfield_expansion} has a simple pointwise counterpart
 $u(x)=\Real(U_f)(x)+o(|x|^{\frac{1-N}{2}})$ as $|x|\to\infty$ and it is expected,    
 as in the second order case, that solutions to~\eqref{intro} should not decay faster than
 $O(|x|^{\frac{1-N}{2}})$.
 So far, however, it is unclear how to prove such a claim even in the radial setting since 
 methods from \cite{MMP} do not seem to be easily generalizable to the fourth order case. Let us remark that
 numerical considerations indicate that radial solutions in the parameter ranges $\alpha<0$ and 
 $\alpha>0,\beta<-2\sqrt\alpha$ (for constant $\Gamma$, say) behave rather differently. While the solutions
 in the former case seem to remain bounded with oscillatory behaviour for all small initial data, the
 solutions in the latter case seem to be unbounded for most initial data so that we expect
 $L^p(\R^N)$-solutions as the ones from Theorem~\ref{mainthm} only at exceptional initial values. In particular, there is little hope 
 to treat this case by the methods from \cite{MMP}.

 \medskip

 The plan of this paper is the following: in Section $2$, we introduce some notation and provide a few
 preliminary results. In particular, the construction of the resolvent-type operators $\mathcal
 R_a$ for the second order case and $\mathfrak R$ for the fourth order operator are explained. 
 In Section $3$, we prove the equivalent of Guti\' errez' a priori estimates (Theorem~6 in~\cite{Gu}) for our
 fourth order operator by proving $(L^p,L^q)$-estimates for $\mathfrak R$. In Section~4 and Section~5 we will prove
 the claims from Theorem~\ref{mainthm}. In Section $4$, using the dual variational approach of \cite{EW}, we
 show the existence of a solution $u\in L^p(\R^N)$ to \eqref{intro}. In Section $5$, qualitative properties of
 this solution such as its regularity and~\eqref{eq:farfield_expansion} will be established.

\section{Preliminaries} \label{sec:preliminaries}

The Fourier transform of a Schwartz function $f\in\mathcal S(\R^N)$ is defined via 
$$
  \hat f(\xi) := \frac{1}{(2\pi)^{N/2}} \int_{\R^N} f(x)e^{-ix\xi}\,dx.
$$
As an $L^2$-isometry, the Fourier transform may be extended to tempered distributions and in
particular to $f\in L^q(\R^N),q\in [1,\infty]$. As in the second order case, we have to construct a
resolvent-type operator $\mathfrak R$ associated with $L=\Delta^2-\beta\Delta+\alpha$. Since this is based on
the corresponding approach to Helmholtz operators via~\eqref{eq:def_a1a2}, let us describe this situation
first.
The fundamental solution $g_a$ of the Helmholtz operator $-\Delta-a,a>0$ is given by
\begin{equation} \label{greenhelm}
  g_a (x)=\frac{i}{4}\Big(\frac{2\pi |x|}{\sqrt{a}}\Big)^{\frac{2-N}{2}} H^{(1)}_{\frac{N-2}{2}}(\sqrt{a}
  |x|),
\end{equation}
see (4.21) in~\cite{Leis}, so that $(-\Delta -a) g_a=\delta$ holds in the distributional sense on $\R^N$.
Here, $H^{(1)}_{(N-2)/2}$ denotes the Hankel function of the first kind of order $(N-2)/2$. From the formulas~9.1.12,9.1.13 and~9.2.1.-9.2.3 in \cite{AbrSte_handbook} we get the following asymptotics:
\begin{align} \label{eq:asymptoticsH}
  \begin{aligned}
  H^{(1)}_{\frac{N-2}{2}}(r)
  &\sim \frac{2}{\sqrt{\pi r}} \Big(e^{i(r-\frac{N-3}{4}\pi)} + O\big(\frac{1}{r} \big)\Big)  
  &&\text{as }r\to \infty, \\
  H^{(1)}_{\frac{N-2}{2}}(r) 
  &\sim \frac{2i}{\pi}\ln \left(\frac{r}{2} \right)
    +O(1)
    &&\text{as }r\to 0^+,\;N=2, \\
  H^{(1)}_{\frac{N-2}{2}}(r) 
  &\sim  -i\sqrt{\frac{2}{\pi r}}  
  +O(r^{1/2})  
   &&\text{as }r\to 0^+,\;N=3, \\
  H^{(1)}_{\frac{N-2}{2}}(r) 
  &\sim -\frac{2i}{\pi r}+ \frac{2i}{\pi} r \ln\left(\frac{r}{2}\right)+ O(r) 
  &&\text{as }r\to 0^+,\;N=4, \\
   H^{(1)}_{\frac{N-2}{2}}(r) 
  &\sim -\frac{\Gamma(\frac{N-2}{2})i}{\pi}
  \left(\frac{2}{r}\right)^{\frac{N-2}{2}}+O(r^{\frac{6-N}{2}})\qquad &&\text{as }r\to 0^+,\;N\geq 5.
  \end{aligned}
\end{align} 
In particular, for all $a>0$ we have the following estimate: 
\begin{equation} \label{asympbessel}
  |g_a(x)| \leq C(|x|^{2-N}+|\log(|x|)|) \quad (0<|x|\leq 1),\qquad
  |g_a(x)| \leq C|x|^{\frac{1-N}{2}} \quad (|x|\geq 1).
\end{equation}
In the case $N\geq 3$ we have $g_0(r)=\frac{1}{(N-2)N\omega_N} r^{2-N}$ where $\omega_N$ denotes the area of
the sphere $S^{N-1}$, see (4.1) in \cite{GiTr}. Hence, \eqref{asympbessel} also holds in the case
$a=0,N\geq 3$, but not for $N=2$ because $g_0(x)=\frac{1}{2\pi}\log(|x|)$ grows logarithmically at infinity.
This is responsible for the extra assumption $N\geq 3$ in the case $\alpha=0,\beta<0$ from assumption (A1).
For positive $a$ the functions $g_a$ are known to satisfy the Sommerfeld radiation condition at infinity:
\begin{equation}\label{eq:Sommerfeld_ga}
  \nabla g_a(x) - i \sqrt{a} g_a(x) \frac{x}{|x|} = O(|x|^{-\frac{N+1}{2}}) \quad\text{as }|x|\to\infty.
\end{equation}
The resolvent-type operator $\mathcal R_a$ associated with $-\Delta-a$ for $a>0$ is then defined via 
\begin{equation} \label{eq:definition_2ndorder_resolvent}
  \mathcal{R}_a f := \lim_{\varepsilon \rightarrow 0^+} \mathcal{R}_{a+i\varepsilon} f,
  \quad\text{where }
  (\mathcal{R}_{a+i\varepsilon} f)(x) :=\frac{1}{(2\pi)^{N/2}} \int_{\R^N}e^{i x \xi} 
  \frac{\hat{f}(\xi)}{|\xi|^2 - (a+i\varepsilon) }\,d\xi. 
\end{equation}
Notice that the same formula holds in the case $a\leq 0$. Here, the limit has to be understood in the $L^p(\R^N)$-sense for $\frac{2(N+1)}{N-1}\leq
p\leq \frac{2N}{N-2}$ whenever $f\in L^{p^\prime}(\R^N)$. In fact, Guti\'errez proved in \cite{Gu}[Theorem 6]
the estimate $\|\mathcal{R}_{a+i\varepsilon} f \|_{L^p (\R^N)}\leq \| f \|_{L^{p^\prime} (\R^N)}$ for all
Schwartz functions $f\in \mathcal{S}(\R^N)$, so that the operator $\mathcal R_a$ is a well-defined bounded
linear operator from $L^{p'}(\R^N)$ to $L^p(\R^n)$ by the Uniform Boundedness Principle.  Moreover, this operator may be expressed
in terms of the fundamental solution $g_a$ from~\eqref{greenhelm} via 
\begin{equation} \label{eq:EWresolvent_vs_ga1}
  (\mathcal{R}_a f)(x)
  =  (g_a \ast f)(x) 
  = \frac{1}{a} \mathcal{R}_1 \Big( f(\frac{\cdot}{\sqrt a})\Big)(\sqrt a x) 
  \qquad\text{for } f\in \mathcal{S}(\R^N),a>0.
\end{equation}

\medskip

In view of~\eqref{eq:def_a1a2} the corresponding quantities for the fourth order operator $L=\Delta^2
-\beta\Delta + \alpha$ may be defined analogously. We put
\begin{align}\label{eq:def_G}
  \begin{aligned}
  G &:=\frac{1}{\sqrt{\beta^2-4\alpha}} (g_{a_1} - g_{a_2}), \quad
  \hat G(\xi) 
  &= \frac{1}{\sqrt{\beta^2-4\alpha}} \Big(
  \frac{1}{|\xi|^2-a_1} - \frac{1}{|\xi|^2-a_2}\Big) \quad( |\xi|\neq a_1,a_2)
\end{aligned}
\end{align} 
we find $\Delta^2 G -\beta \Delta G +\alpha G= \delta$ in the distributional sense on $\R^N$ so that $G$ is a
fundamental solution of $L$.
Notice that formally the same definition has been used in \cite{BCDN}[Proposition $3.13$] when $a_1,a_2<0$,
while our focus lies on the cases $a_1>0>a_2$ or $a_1>a_2>0$ or $a_1=0>a_2$. 
From \eqref{greenhelm}--\eqref{asympbessel} we deduce (taking into account the cancellations at zero)
\begin{align} \label{asymgreen}
  \begin{aligned}
  |G(x)| &\leq \begin{cases}
    C(|x|^{4-N} +|\log(|x|)|) &, N\geq 4 \\
    C &,N \in\{2,3\}
  \end{cases}\quad &&(0<|x|\leq 1),\\
   |G(x)| &\leq C|x|^{(1-N)/2}  && (|x|\geq 1)
\end{aligned}
\end{align}
Moreover, from~\eqref{eq:Sommerfeld_ga} we get that $G$ satisfies a variant of Sommerfeld's outgoing radiation
condition (see \cite{CoDa}) given by  
\begin{align}   \label{sommerfeld}
  \begin{aligned}
   |\nabla G (x) - i \sqrt{a_1} G(x) \frac{x}{|x|} |&=o(|x|^{\frac{1-N}{2}})\ \text{as } |x|\rightarrow
  \infty ,\text{if }a_1>0>a_2,\\
  |\nabla G (x) - \dfrac{i}{a_1-a_2}(\sqrt{a_1}g_{a_1}(x)-\sqrt{a_2} g_{a_2}(x) )\frac{x}{|x|}|
 & =o(|x|^{\frac{1-N}{2}})\ \text{as } |x|\rightarrow \infty,\ \text{if}\ a_1> a_2>0.
\end{aligned}
\end{align}
Notice that in the case $a_1>0>a_2$ the functions $g_{a_2},g_{a_2}'$ decrease
exponentially and hence much faster than $g_{a_1},g_{a_1}'$ at infinity so that \eqref{eq:def_G} allows to
deduce the Sommerfeld condition from~\eqref{eq:Sommerfeld_ga}. Here, $g_{a_2}$ denotes the
Green's function of the Schr\"odinger operator $-\Delta-a_2$. As to the case $a_1>a_2>0$ let us remark
$\sqrt{\beta^2-4\alpha}=a_1-a_2$. Motivated by \eqref{eq:definition_2ndorder_resolvent}--\eqref{eq:def_G}
we may now define  
\begin{equation}\label{eq:def_resolvent}
  \mathfrak {R} f 
  := \lim_{\varepsilon \rightarrow 0}  \frac{1}{\sqrt{\beta^2-4\alpha}}\big( \mathcal
  R_{a_1+i\eps}f - \mathcal R_{a_2+i\eps} f\big).  
\end{equation}
Being interested in real-valued solutions of~\eqref{intro} we will need $\bold R:=\Real(\mathfrak R)$.
We will show in Theorem~\ref{thminvopbound} that $\mathfrak{R}$ is a (complex-valued) bounded linear operator between 
$L^p(\R^N)$ and $L^q(\R^N)$ such that $\bold{R} f$ defines a (real-valued) distributional solution
of $\Delta^2 u -\beta \Delta u +\alpha u =f$ provided $f\in L^p(\R^N)$. Actually, better
qualitative properties of $u$ will be shown in
Section~\ref{sec:qualitative}. From~\eqref{eq:def_resolvent} and arguing as in Lemma~4.1 in \cite{EW} 
we get
\begin{align} \label{eq:resolvent_symmetry}
  \int_{\R^N}  (\bold{R} f) g\,dx
  = \int_{\R^N}   f (\bold{R} g)\,dx
 \end{align}
for all $f,g \in \mathcal{S}(\R^N)$. In the next section we provide the $(L^p,L^q)$-estimates for $\mathfrak
R,\bold R$. For notational convenience, the symbol $C$ will stand for a positive number that may change from
line to line.

\section{Resolvent estimates}

In this section we investigate the continuity properties of the resolvent $\mathfrak R$ as an operator
between Lebesgue spaces on $\R^N$. As in the paper by Ev\'equoz and Weth \cite{EW} these properties turn out to
be crucial for proving the existence of solutions of~\eqref{4nls} via a dual variational approach, which we
will set up in the next section. As in the proof of Theorem~6 in~\cite{Gu} the continuity properties are
established via interpolation and the Stein-Tomas theorem (see \cite{Tom_A_restriction} and p.386
in~\cite{Ste_harmonic}). Denoting by $S^{N-1}$ the unit sphere in $\R^N$ and by $\sigma$ the canonical surface
measure on $S^{N-1}$, this theorem reads as follows.

\begin{thm}[Stein-Tomas] \label{STT}
  Let $1\leq p\leq \frac{2(N+1)}{N+3}$. Then there is a $C>0$ such that for all $g\in \mathcal S(\R^N)$ the
  following inequality holds: 
  \begin{equation*} 
    \Big(\int_{S^{N-1}} |\hat{g}(r\omega)|^2 \,d\sigma(\omega)\Big)^{1/2} 
    \leq C  r^{-N(1-1/p)}\|g\|_{L^p(\R^N)}.
  \end{equation*} 
\end{thm}

Notice that the Stein-Tomas inequality for $r\neq 1$ follows from the classical one ($r=1$) by rescaling.
Moreover, we will use the Riesz-Thorin interpolation theorem, see for instance Theorem~1.3.4 in
\cite{Grafakos}.

\begin{thm}[Riesz-Thorin] \label{thm:Riesz_Thorin}
If $T: L^{p_0}(\R^N)+L^{p_1}(\R^N)\to L^{q_0}(\R^N)+L^{q_1}(\R^N)$ such that $\|T\|_{L^{p_0}(\R^N) \to
L^{q_0}(\R^N)}\leq M_0$ and $\|T\|_{L^{p_1}(\R^N)\to  L^{q_1}(\R^N)}\leq M_1$, then, we have 
$$
  \|T\|_{L^p(\R^N) \to L^q(\R^N)}\leq M_0^{1-\theta}M_1^\theta
$$ 
provided that 
$$
  \frac{1}{p}= \frac{1-\theta}{p_0} + \frac{\theta}{p_1} \quad \text{ and } \quad 
  \frac{1}{q}= \frac{1-\theta}{q_0} + \frac{\theta}{q_1}. 
$$ 
\end{thm}

With these preliminary results at hand we are now in the position to prove the resolvent estimates. We 
closely follow the proof of Theorem~6 in Guti\'{e}rrez' paper~\cite{Gu} along with its generalizations from
Theorem~2.1 in~\cite{Ev}. 
 
\begin{thm}  \label{thminvopbound}
  Assume (A1). Then the operator $\mathfrak{R}$ defined by \eqref{eq:def_resolvent} extends to a
  bounded linear operator $\mathfrak{R}:L^p(\R^N)\to L^q(\R^N)$, i.e. 
  \begin{equation} \label{eq:resolvent_estimate}
    \|\mathfrak{R} f \|_{L^q(\R^N)}
    \leq C \|f\|_{L^p (\R^N)},
  \end{equation}
  provided that $p,q\in [1,\infty]$ satisfy
  \begin{equation}
  \label{condgutie}
	\frac{2}{N+1}\leq \frac{1}{p}-\frac{1}{q} \begin{cases}
	  \leq 1 &,\text{if } N\in\{2,3\}\\
	  < 1 &,\text{if } N=4 \\
	  \leq \frac{4}{N} &,\text{if } N\geq 5\\
	\end{cases}, \quad\; 
	\frac{1}{p}> \frac{N+1}{2N},\quad\; \frac{1}{q}<\frac{N-1}{2N}.
  \end{equation}
  In particular, \eqref{eq:resolvent_estimate} holds for $q=p^\prime$ whenever
  $\frac{2(N+1)}{N-1}\leq q\leq \infty$ for $N\in\{2,3\}$, $\frac{10}{3}\leq q<\infty$ for $N=4$ or
  $\frac{2(N+1)}{N-1}\leq q\leq \frac{2N}{N-4}$ for $N\geq 5$.
\end{thm}
\begin{proof}
  We only deal with the case $a_1>0\geq a_2$. Recall that $a_1,a_2$ were defined in \eqref{eq:def_a1a2}.  We
  will comment on the necessary modifications in the case $a_1>a_2>0$ at the end of the proof. We split the operator $f\mapsto \mathfrak{R}f=G\ast f$ into a resonant
  and a nonresonant part.
  To this end let $\psi \in \mathcal{S}(\R^N)$ be a function such that $\hat{\psi}\in C_c^\infty(\R^N)$ satisfies
$0\leq \hat{\psi}\leq 1$ and
\begin{equation} \label{eq:def_psi}
  \hat{\psi} (\xi)=\begin{cases}
    1 &,\text{if}\ ||\xi|-\sqrt a_1|\leq \frac{\sqrt a_1}{6},\\ 
    0 &,\text{if}\ ||\xi|-\sqrt a_1|\geq \frac{\sqrt a_1}{4}
    \end{cases}. 
\end{equation}
Next define $G_1 := \psi \ast G$ and $G_2:=G -
G_1=(1-\psi)\ast G$.  First we establish pointwise bounds for $G_1$ and $G_2$. 
By \eqref{asymgreen} we know that $|G(x)|\leq C |x|^{\frac{1-N}{2}}$ for $|x|\geq 1$ so that  
$G_1=\psi\ast G$ and  $\psi\in\mathcal{S}(\R^N)$ imply
\begin{equation} \label{gutiee1}
  |G_1 (x)|\leq C (1+|x|)^{\frac{1-N}{2}} \quad\text{for all }x\in\R^N.
\end{equation}
Furthermore, thanks to \eqref{asymgreen} and \eqref{gutiee1}, we deduce   
$$
  |G_2(x)|\leq \begin{cases}
     C |x|^{4-N} &, \text{if}\ N> 4\\ 
     C(1+|\log|x||)&, \text{if}\ N=4 \\
 C&, \text{if}\ N=\{2,3\}, \\
    \end{cases}
  \qquad\text{for }|x|\leq 1.
$$
 Since $\hat{G}_2= (1-\hat \psi ) \hat G$ and $\hat G$ is given by \eqref{eq:def_G}, we find 
 $\partial^{\gamma} \hat{G}_2 \in L^1 (\R^N)$ for all multi-indices $\gamma \in \N_0^N$ such that
 $|\gamma|\geq N-3$ if $a_2>0$ resp. $|\gamma|\in \{N-3, N-2, N-1\}$ if $a_2=0$. Hence, $|G_2(x)|\leq C_s
 |x|^{-s}$ for all $s\geq N-3$ in the case $a_2>0$ whereas $|G_2(x)|\leq C_s
 |x|^{-s}$ for $N-3\leq  s\leq  N-1$ in the case $a_2=0$. From this we deduce  
\begin{equation}
\label{estG2}
|G_2 (x)|\leq \begin{cases} C \min \{|x|^{4-N},  |x|^{-N}\} &,\text{if}\ N> 4,\\ 
  C \min \{1+ |\log|x|||, |x|^{-N}\} &,\text{if}\ N= 4,\\
C \min \{1,|x|^{-N} \}&, \text{if}\ N=\{2,3\}, \\
\end{cases}
  \qquad\text{for all }x\in\R^N,\;a_2>0  
\end{equation}
as well as
\begin{equation}
\label{estG2a_2=0}
|G_2 (x)|\leq \begin{cases} C \min \{|x|^{4-N},  |x|^{-N-1}\} &,\text{if}\ N> 4,\\ 
  C \min \{1+ |\log|x||, |x|^{-N-1}\} &,\text{if}\ N= 4,\\
C \min \{1,|x|^{-N-1} \}&, \text{if}\ N=3, \\
\end{cases}
  \qquad\text{for all }x\in\R^N,\;a_2=0.
\end{equation}

\medskip

We now use these pointwise bounds for $G_2$ in order to prove that the nonresonant part satisfies the mapping
properties asserted above. For $p,q$ as in \eqref{condgutie} we define 
$r\in (\frac{N+1}{N-1},\frac{N}{(N-4)_+})$ via $1+\frac{1}{q}=\frac{1}{r}+\frac{1}{p}$ so that Young's
convolution inequality and $G_2\in L^r(\R^N)$ gives 
\begin{equation}\label{RT0}
  \|G_2 \ast f \|_{L^q(\R^N )} \leq \|G_2\|_{L^r (\R^N )} \|f\|_{L^p(\R^N )} \leq C \|f\|_{L^p(\R^N )}.
\end{equation}
In the limit case $\frac{1}{p}-\frac{1}{q}=1$ and $N\in\{2,3\}$ this inequality follows the same way using 
$G_2\in L^\infty(\R^N)$, see \eqref{estG2}. In the limit case $\frac{1}{p}-\frac{1}{q}=\frac{4}{N}$ and
$N\geq 5$ it follows from $G_2\in L^{N/(N-4),w}(\R^N)$ and Young's inequality for weak Lebesgue
spaces, see Theorem~1.4.24 in~\cite{Grafakos}. 

\medskip 

Next we estimate the resonant term $G_1 \ast f$. To this end let $\eta \in C_{c}^\infty (\R^N) $ be a cut-off
function such that $\eta(x)=1$ for $|x|\leq 1$ and $\eta(x)=0$ if $|x|\geq 2$. For $j\in \N$ we define
$\eta_j(x):=\eta (x/2^j ) - \eta (x /2^{j-1})$ and $\eta_0 :=\eta$. As a
consequence, 
\begin{equation} \label{eq:def_G1j}
  G_1=\sum_{j=0}^\infty G_1^j\quad \text{with}\; G_1^j:=G_1 \eta_j \text{ so that }
  |G_1^j(x)| \leq C 2^{j(1-N)/2}1_{[2^{j-1},2^{j+1}]}(|x|),
\end{equation}
where the latter estimate follows from \eqref{gutiee1}. Next, let $\varphi \in \mathcal{S}(\R^N)$ be chosen
such that 
\begin{equation}\label{eq:def_varphi}
  \hat \varphi (\xi )=
  \begin{cases}1 &,\text{if}\ ||\xi|-\sqrt a_1|\leq \sqrt a_1/2,\\ 
 0 &,\text{if}\ ||\xi |-\sqrt a_1 |\geq 3\sqrt a_1/4.\end{cases}
\end{equation} 
Notice that this definition guarantees $\supp(\hat{G_1})\subset \{\xi\in\R^N: \hat\varphi(\xi) = 1\}$ so that
the following identity holds by the definition of $G_1$ and \eqref{eq:def_psi}: 
 \begin{equation}\label{eq:def_Qj}
   G_1\ast f = (G_1\ast \varphi) \ast f= \sum_{j=0}^\infty Q^j\ast f\qquad \text{with}\ Q^j =G_1^j \ast
   \varphi.
 \end{equation}
 
 \medskip
 
Using  Plancherel's Theorem and the Stein-Tomas Theorem (see Theorem~\ref{STT}) we get for all
$f\in\mathcal{S}(\R^N)$ and $g:=\varphi \ast f$
\begin{align} \label{RT2}
  \begin{aligned}
  \|Q^j \ast f\|_{L^2 (\R^N )}^2
  &= \|G_1^j\ast g\|_{L^2(\R^N)}^2   \\
  &= \int_{||\xi| -\sqrt a_1| \leq 3 \sqrt a_1 /4} |\hat{G}_1^j (\xi ) \hat{g}(\xi )|^2 \,d\xi  \\
  &\leq C \int_{\sqrt a_1/4}^{7\sqrt a_1 /4}r^{N-1} |\hat{G}_1^j (r )|^2 \int_{S^{N-1}}
  |\hat{g}(r \omega )|^2 d\sigma(\omega) \,dr \\
  &\leq C\int_{\sqrt a_1/4}^{7\sqrt a_1/4} r^{N-1}|\hat{G}_1^j (r)|^2 \cdot r^{-2N(1-\frac{N+3}{2(N+1)})}\|g\|_{L^{ 2(N+1)/(N+3)}(\R^N )}^2
  \,dr \\
  &\leq C\|\hat{G}_1^j\|_{L^2(\R^N)}^2 \|g\|_{L^{ 2(N+1)/(N+3)}(\R^N )}^2 \\
&\leq C \|G_1^j\|_{L^2(\R^N)}^2\|\varphi\|_{L^1 (\R^N )}^2\|f\|_{L^{ 2(N+1)/(N+3)}(\R^N )}^2  \\
&\leq C 2^j \|f\|_{L^{ 2(N+1)/(N+3)}(\R^N )}^2.
  \end{aligned}
\end{align}
In the last inequality we estimated the $L^2$-norm of $G_1^j$ by exploiting \eqref{eq:def_G1j}. Furthermore, we derive the inequality
\begin{align}  \label{RT1}
  \begin{aligned}
  \|Q^j \ast f\|_{L^{\tilde q}(\R^N)}
  &\leq \|Q^j\|_{L^r(\R^N )} \|f\|_{L^{\tilde p}(\R^N)} \\ 
  &\leq \|\varphi \|_{L^1(\R^N )} \|G_1^j \|_{L^r(\R^N )}\|f\|_{L^{\tilde p}(\R^N)} \\ 
  &\leq C 2^{j((1-N)/2 + N/r)}\|f\|_{L^{\tilde p}(\R^N)} \\
  &= C 2^{j((1+N)/2 +N/{\tilde q} - N/{\tilde p})}\|f\|_{L^{\tilde p}(\R^N)}
  \qquad\text{if }1\leq \tilde p\leq \tilde q\leq \infty,
  \end{aligned}
\end{align}
and $r\in [1,\infty]$ is defined according to $1+\frac{1}{\tilde q} = \frac{1}{r}+\frac{1}{\tilde p}$. Notice
that in the third inequality we estimated $\|G_1^j\|_{L^r(\R^N)}$ once again by exploiting \eqref{eq:def_G1j}.
Interpolating the estimates \eqref{RT2} and
\eqref{RT1} yields
$$
  \|Q^j \ast f\|_{L^q(\R^N)}
  \leq C 2^{j(1/2 +\theta N(1/2+1/{\tilde q} - 1/{\tilde p}))}\|f\|_{L^p(\R^N)}
  \qquad (j\in\Z)
$$
provided $\frac{1}{p}=\frac{\theta}{\tilde p}+\frac{(1-\theta)(N+3)}{2(N+1)}$ and
$\frac{1}{q}=\frac{\theta}{\tilde q}+\frac{1-\theta}{2}$ with $\theta\in [0,1],\tilde p,\tilde q\in
[1,\infty]$. Substituting $\tilde q$ yields
$$
  \|Q^j \ast f\|_{L^q(\R^N)}
  \leq C 2^{j( (1-N)/2 + N/q +\theta N(1- 1/{\tilde p}))}\|f\|_{L^p(\R^N)}
  \qquad (j\in\Z)
$$
whenever $\frac{1}{p}=\frac{\theta}{\tilde p}+\frac{(1-\theta)(N+3)}{2(N+1)}$ for some $\tilde p\in
[1,\infty],\theta\in [1-\frac{2}{q},1],q\geq 2$. Substituting now $\tilde p$ gives
$$
  \|Q^j \ast f\|_{L^q(\R^N)}
  \leq C 2^{j( (3N+1)/2(N+1) + N/q - N/p +\theta N(N-1)/2(N+1))}\|f\|_{L^p(\R^N)}
  \qquad (j\in\Z)
$$
for $1\geq \theta\geq \max\{1-\frac{2}{q},\frac{2(N+1)-p(N+3)}{(N-1)p}\}$ for $q\geq 2$ and $1\leq p\leq
2(N+1)/(N+3)$. Summing up these estimates we get
\begin{align} \label{RT3} 
  \|G_1\ast f\|_{L^q(\R^N )} 
  \leq C \|f\|_{L^p(\R^N)} \quad\text{provided } 
  1\leq p\leq \frac{2N(N+1)}{N^2+4N-1},\;\frac{2N}{N-1}<q\leq \infty.
\end{align} 
By duality, the $(L^p,L^q)$-estimate provides the corresponding $(L^{q'},L^{p'})$-estimate that reads
\begin{align} \label{RT4} 
  \|G_1\ast f\|_{L^q(\R^N )} 
  \leq C \|f\|_{L^p(\R^N)} \quad\text{provided }  
  1\leq p< \frac{2N}{N+1},\;\frac{2N(N+1)}{(N-1)^2}\leq q\leq \infty.
\end{align} 
By \eqref{RT3},\eqref{RT4} the estimates for $G_1\ast f$ hold for all $p,q$ as in \eqref{condgutie} under
the additional assumption $1\leq p\leq \frac{2N(N+1)}{N^2+4N-1}$ or $\frac{2N(N+1)}{(N-1)^2}\leq q\leq\infty$. For all other exponents
$p,q$ as in \eqref{condgutie} except on the line $\frac{1}{p}-\frac{1}{q}=\frac{2}{N+1}$ we may choose 
\begin{equation}\label{RT6}
  p_1\in \Big[1,\frac{2N(N+1)}{N^2+4N-1}\Big],\;q_1\in \Big(\frac{2N}{N-1},\infty\Big],\qquad 
  p_2\in \Big[1,\frac{2N}{N+1}\Big),\;q_1 \in \Big[\frac{2N(N+1)}{(N-1)^2},\infty\Big]
\end{equation} 
such that
\begin{equation}\label{RT7}
   \frac{1}{p_1}-\frac{1}{q_1} = \frac{1}{p_2}-\frac{1}{q_2} = \frac{1}{p}-\frac{1}{q}\in
   \Big(\frac{2}{N+1},1\Big].
 \end{equation}
 Then, by \eqref{RT3},\eqref{RT4} the operator $G_1\ast f$ is bounded from $L^{p_j}(\R^N)$ to $L^{q_j}(\R^N)$
 for $j=1,2$ and we have $p_1<p<p_2$. This follows from \eqref{RT6},\eqref{RT7} since we have assumed $p$
 to satisfy neither  $1\leq p\leq \frac{2N(N+1)}{N^2+4N-1}$ nor $\frac{2N(N+1)}{(N-1)^2}\leq q\leq\infty$.
 In particular, by~\eqref{RT7} we can find $\theta\in (0,1)$ such that
 $\frac{1}{p}=\frac{\theta}{p_1}+\frac{1-\theta}{p_2}$ and hence
 $\frac{1}{q}=\frac{\theta}{q_1}+\frac{1-\theta}{q_2}$. So the Riesz-Thorin Theorem  finally yields
\begin{equation} \label{RT5}
  \|G_1\ast f\|_{L^q(\R^N)} 
  \leq C \|f\|_{L^p(\R^N)} \quad
  \text{if}\quad
  \frac{2}{N+1}<\frac{1}{p}-\frac{1}{q}\leq 1,\;
  \frac{1}{p}>\frac{N+1}{2N},\;\frac{1}{q}<\frac{N-1}{2N}. 
\end{equation}
Hence, the assertion of the theorem for $1/q-1/p>2/(N+1)$ follows from \eqref{RT0}
and~\eqref{RT5}.

\medskip

We finally consider the missing limiting case $1/q-1/p=2/(N+1)$, which is
established using Lorentz space interpolation following the ideas of Guti\'{e}rrez, see p.20 in~\cite{Gu}. As
in section~5.3 of \cite{SteinWeiss} we denote by $\|\cdot\|_{p,s}$ the standard norm of the Lorentz space $L^{p,s}(\R^N)$ so that the problem
reduces to proving
\begin{equation} \label{lorentze1}
  \|G_1 \ast f\|_{q ,\infty}\leq C\|f\|_{p,1},
\end{equation}
for $(p,q)=(p_0,q_0)$ and for $(p,q)=(q_0^\prime , p_0^\prime )$ where $q_0 = 2N/(N-1)$ and $p_0= 2N(N+1)/(N^2
+4N -1)$.
So let $E\subset\R^N$ be any measurable set of finite measure and for any given $\lambda>0$ we define
$A:=\{x\in \R^N : |(G_1 \ast 1_E) (x) |>\lambda \}$. By Theorem~3.13 in section~5 in~\cite{SteinWeiss} we see that
\eqref{lorentze1} is equivalent to
\begin{equation} \label{lorentze2}
  \lambda |A|^{1/q}\leq C |E|^{1/p}\quad\text{whenever }\lambda>0.
\end{equation}
 In view of the definition of $Q^j$ from~\eqref{eq:def_Qj} and the estimates \eqref{RT2},\eqref{RT1}, we have
 for all $M\in\Z$
\begin{align*}
  \lambda|A|
  &\leq \int_A |(G_1 \ast 1_E ) (x)| \,dx \\ 
  &\leq \sum_{j=0}^\infty \int_A |(Q^j \ast 1_E )(x)| \,dx \\
  &\leq \sum_{j=0}^\infty  \big( \|Q^j \ast 1_E \|_{L^2 (\R^N)} |A|^{1/2} 1_{j\leq M}
  + \|Q^j \ast 1_E\|_{L^\infty (\R^N)}  |A| 1_{j\geq M+1} \big) \\
  &\leq \sum_{j=0}^\infty  \big( 2^{j/2} |E|^{\frac{N+3}{2(N+1)}} |A|^{1/2} 1_{j\leq M}
  + 2^{j(1-N)/2} |E| |A| 1_{j\geq M+1} \big) \\
  &\leq |E|^{\frac{N+3}{2(N+1)}} |A|^{1/2}\cdot \sum_{j=-\infty}^M 2^{j/2}  
  + |E| |A| \cdot \sum_{j=M+1}^\infty 2^{j(1-N)/2}  \\
  &\leq C\left( 2^{M/2} |E|^{\frac{N+3}{2(N+1)}} |A|^{1/2}+ 2^{(M+1)(1-N)/2}|E||A|\right).
\end{align*}
Choosing $M\in\Z$ such that 
$$
  |E|^{\frac{N-1}{N(N+1)}}|A|^{1/N}\leq 2^M < 2 |E|^{\frac{N-1}{N(N+1)}}|A|^{1/N}
$$
we deduce \eqref{lorentze2} for $(p,q)=(p_0,q_0)$.   

\medskip
 
Finally, we consider the case $(p,q)=(q_0^\prime , p_0^\prime)$. Using 
the dual version of the inequality~\eqref{RT2} we get similarly as above
\begin{align*}
  \lambda|A|
  &\leq \sum_{j=0}^\infty \|Q^j \ast 1_E \|_{L^{\frac{2(N+1)}{N-1}}(\R^N )}
 |A|^{\frac{N+3}{2(N+1)}}  1_{j\leq M} +\sum_{j=0}^\infty \|Q^j \ast 1_E\|_{L^\infty (\R^N )} |A| 1_{j\geq
 M+1}\\
  & \leq C \big( 2^{M/2} |E|^{1/2} |A|^{\frac{N+3}{2(N+1)}}+ 2^{-(M+1) (N-1) /2}|E||A| \big) \\
  &\leq C |E|^{\frac{N+1}{2N}} |A|^{\frac{N^2+4N-1}{2N(N+1)}}
\end{align*}
where $M\in\Z$ was chosen according to
$$
  |E|^{\frac{1}{N}} |A|^{\frac{N-1}{N(N+1)}}\leq 2^M < 2 |E|^{\frac{1}{N}} |A|^{\frac{N-1}{N(N+1)}}.
$$
So we obtain \eqref{lorentze2} for $(p,q)=(q_0',p_0')$, which finishes the proof of
Theorem~\ref{thminvopbound} under assumption $a_1>0\geq a_2$. 

\medskip

It remains to discuss the modifications for the case $a_1>a_2>0$. One defines the functions $G_1,G_2$ as
above, but with a function $\psi\in\mathcal{S}(\R^N)$ satisfying \eqref{eq:def_psi} both for $a_1$ and for
$a_2$, which is possible due to $a_1>a_2>0$. Accordingly, the function $\varphi\in\mathcal S(\R^N)$ has to be
chosen such that \eqref{eq:def_varphi} holds both for $a_1$ and for $a_2$. Arguing as above yields the
resolvent estimates.

\end{proof} 

\begin{rmq} \label{rmqradial} 
  We show how our results may be improved in the radial setting where G. Ev\'equoz \cite{Evrad} recently
  announced the following inequality
  \begin{equation}\label{eq:restriction_conjecture}
    \|\hat f\|_{L^\infty_{rad}(S^{N-1})}\leq C\|f\|_{L^r_{rad}(\R^N)}\quad\text{for } 1\leq r<\frac{2N}{N+1}
  \end{equation}
  with a constant depending on $r$. Using this estimate in \eqref{RT2} instead of the Stein-Tomas Theorem,
  we get $\|Q^j \ast f\|_{L^2_{rad} (\R^N )}^2\leq C 2^j \|f\|_{L^r_{rad}(\R^N )}^2$  for such $r$.
  Interpolating this estimate with \eqref{RT1} for $\tilde{q}=\infty,\tilde p=1$, we find  
  for all $\theta\in [0,1],r\in [1,\frac{2N}{N+1})$ the inequality
  $$
    \|Q_j\ast f\|_{L^q_{rad}(\R^N)}
    \leq C 2^{j(\theta/2+(1-\theta)(1-N)/2)} \|f\|_{L^p_{rad}(\R^N)}
    \quad\text{if }
    \frac{1}{q}=\frac{\theta}{2}+\frac{1-\theta}{\infty},\frac{1}{p}=\frac{\theta}{r}+\frac{1-\theta}{1}. 
  $$
  So one finds $\theta=2/q$ and thus
  \begin{equation} \label{RT5bis}
    \|Q_j\ast f\|_{L^q_{rad}(\R^N)} 
    \leq C 2^{j/2( 1-N+2N/q)}  \|f\|_{L^p_{rad}(\R^N)} \quad
    \text{if } 
    \frac{N}{N-1}(1-\frac{1}{p})<\frac{1}{q}\leq 2,\;
    \frac{N+1}{2N}<\frac{1}{p}\leq 1. 
  \end{equation}
  Summing over $j\in\N_0$ one obtains
  \begin{equation}\label{RT5bisbis}
    \|G_1\ast f\|_{L^q_{rad}(\R^N)}   
    \leq C \|f\|_{L^p_{rad}(\R^N)}
    \quad\text{if }
    \frac{N}{N-1}(1-\frac{1}{p})<\frac{1}{q}<\frac{N-1}{2N},\;
    \frac{N+1}{2N}<\frac{1}{p}\leq 1. 
  \end{equation}
  Interpolating between this inequality and its dual, we find
  $$ 
    \|G_1\ast f\|_{L^q_{rad}(\R^N)} 
    \leq C \|f\|_{L^p_{rad}(\R^N)}
    \quad\text{if }  1\geq  \frac{1}{p}>\frac{N+1}{2N},\; 
      \frac{1}{q}< \frac{N-1}{2N},\; \frac{1}{p}-\frac{1}{q}>\frac{3N-1}{2N^2}.
  $$ 
  Indeed, for any given such pair $(p,q)$ one may choose $Q\in (2N/(N-1),q)$ such that
  $1/p-1/q > 1-(2N-1)/(QN)>(3N-1)/(2N^2)$ and then define $P$ via 
  $1/p-1/q=:1/P-1/Q$. Then the couple $(P,Q)$ satisfies \eqref{RT5bisbis} and interpolating the $(L^P,L^Q)$
  with the weight $\theta:=(\frac{1}{p}+\frac{1}{Q}-1)/(\frac{1}{P}+\frac{1}{Q}-1) \in [0,1]$ and the
  dual $(L^{Q'},L^{P'})$-estimate with the weight $1-\theta$ gives the desired estimate. In particular, the
  assumption $\frac{1}{p}-\frac{1}{q}>\frac{2}{N+1}$ from~\eqref{condgutie} may be replaced by the weaker
  assumption $\frac{1}{p}-\frac{1}{q}>\frac{3N-1}{2N^2}$ and our existence result from Theorem~\ref{mainthm}
  extends to all exponents $p\in (\frac{2N}{N-1},\frac{2N}{N-4})$ when $\Gamma$ is a positive constant (so
  that the radial setting is meaningful).
\end{rmq}

\section{Existence of solutions via dual variational methods}

In this section we prove the existence of a nontrivial solution to 
$$
  \Delta^2 u -\beta \Delta u  +\alpha u = \Gamma |u|^{p-2}u\quad\text{in }\R^N.
$$
We proceed
along the lines of Theorem~1.1 in \cite{EW} and Theorem~1.3 in \cite{Ev}. Since the proofs are very
similar, we keep the presentation short and refer to the corresponding results in \cite{Ev,EW} when
necessary. Adopting a dual variational approach we look for a function $v:=\Gamma^{1/p^\prime} |u|^{p-2}u\in L^{p^\prime}(\R^N)$
satisfying the following integral equation
\begin{equation} \label{eq:dual_equation}
 \Gamma^{1/p} \bold  R ( \Gamma^{1/p} v ) = |v|^{p^\prime -2}v\quad\text{in }\R^N
\end{equation} 
whenever $\frac{2(N+1)}{N+3}<p<\frac{2N}{(N-4)_+}$. Recall that $\bold R$ is the real part of the
complex resolvent $\mathfrak R$ of the fourth order linear operator appearing in the equation,
see~\eqref{eq:def_resolvent}.
Notice that equation~\eqref{eq:dual_equation} is the Euler-Lagrange equation associated with the functional
$J\in C^1 (L^{p^\prime}(\R^N),\R)$ given by
\begin{equation} \label{eq:def_J}
  J(v) := \frac{1}{p^\prime} \int_{\R^N}  |v|^{p^\prime} \,dx - \frac{1}{2} \int_{\R^N}   \Gamma^{1/p} v \bold{R}(  \Gamma^{1/p}  v)
  \,dx.
\end{equation} 
This is a consequence of the following result:

\begin{prop} \label{lemcompact}
  Assume (A1),(A2). Then $\bold{R}:L^{p^\prime}(\R^N)\rightarrow L^p(\R^N )$ satisfies 
  $$
    \int_{\R^N} w \bold{R} (v)\,dx = \int_{\R^N} v \bold{R}(w)\,dx 
    \quad\text{for all }v,w\in L^{p^\prime}(\R^N).
  $$
  Moreover, 
  for any bounded and measurable set $B\subset\R^N$ the operator $1_B \bold{R}:L^{p^\prime}(\R^N )\rightarrow
  L^p(\R^N)$ is compact.
\end{prop}
\begin{proof}
  We argue as in the proof of \cite{EW}[Lemma~4.1]. The selfdual mapping properties of $\bold{R}$ result
  from Theorem~\ref{thminvopbound} and its symmetry follows from~\eqref{eq:resolvent_symmetry}.
  In order to prove the compactness property let $(v_n)\subset L^{p^\prime}(\R^N)$ satisfy $v_n
  \rightharpoonup 0$ in $L^{p^\prime}(\R^N)$ as $n\rightarrow \infty$. The boundedness and symmetry of
  $\bold R$ yield $\bold{R}(v_n)\rightharpoonup 0$ in $L^p(\R^N)$ as $n\rightarrow \infty$.
  On the other hand, we will show in Proposition~\ref{propreg} that for all $R>0$ we have
  $$
    \|\bold{R}(v_n) \|_{W^{4,p^\prime}(B_R)}
    \leq C_R \big( \|\bold{R} (v_n) \|_{L^{p^\prime} (\R^N)} +\|v_n \|_{L^{p^\prime} (\R^N)} \big) 
    \leq C_R.
  $$ 
  Using the compactness of the embedding $W^{4,p^\prime}(B_R) \hookrightarrow L^p(B_R)$ we find a subsequence
  denoted again by $(v_n)$ such that $\bold{R}(v_n)$ converges in $L^p(B_R)$
  towards its weak limit, which is 0 as we proved above. Since $R>0$ was arbitrary, this finishes the
  proof.
\end{proof}

The following result shows that $J$ has the mountain pass geometry and that it admits a bounded
Palais-Smale sequence at its mountain pass level which, as usual, is defined as
follows: 
\begin{equation}\label{eq:MP_level}
  c = \inf_{\gamma \in P} \max_{t\in [0,1]}J(\gamma (t)). 
\end{equation}
Here, $P= \{\gamma \in C([0,1], L^{p^\prime}(\R^N)):\ \gamma (0)=0 \text{ and } J(\gamma (1))<0
\}$.

\begin{lem} \label{PSbounded}
  Assume (A1),(A2). Then we have:
\begin{itemize}
  \itemsep-2pt 
\item[(i)] There exist $\delta >0,\rho\in (0,1)$ such that $J(v) \geq \delta >0$ for all $v\in
L^{p^\prime}(\R^N )$ with $\|v\|_{L^{p^\prime}(\R^N)}=\rho$.
\item[(ii)] There exists $v_0 \in L^{p^\prime}(\R^N)$ such that $\|v_0\|_{L^{p^\prime}(\R^N)}>1$ and $J(v_0)
<0$.
\item[(iii)] There exists a bounded Palais-Smale sequence $(u_n) \subset L^{p^\prime}(\R^N)$ for $J$ at the
level $c>0$.
\end{itemize}
\end{lem}
\begin{proof}  
  The part (i) is proved exactly as in Lemma~4.2 in~\cite{EW}, see also p.9 in~\cite{Ev}. For the part (ii) we
  argue as in Lemma~3.1 in~\cite{MMP}. We find a $z\in L^{p^\prime}(\R^N)$ such that 
  $$
    \int_{\R^N}  z \bold R z \,dx >0,
  $$ 
  so that $v_0=tz$ for $t$ sufficiently large is a valid choice by~\eqref{eq:def_J}. Indeed, using the
  characterization~\eqref{eq:def_resolvent}, we choose $z\in\mathcal{S}(\R^N)$ such that
  $|\xi|^2>a_1$ for all $\xi\in \supp(\hat z)$ so that $a_1>a_2$ gives
  \begin{align*}
    \int_{\R^N} z \bold R z \,dx
    = \int_{\supp(\hat z)}  \frac{|\hat z(\xi)|^2}{(|\xi|^2-a_1)(|\xi|^2-a_2)} \,d\xi
    > 0.
  \end{align*}
  Part~(iii) is based on the deformation lemma and the proof is the same as the one of Lemma~4.2~(iii) and
  Lemma~6.1 in~\cite{EW}.  
\end{proof}
 
Next we need a compactness property for Palais-Smale sequences obtained in part (iii) of the previous lemma.
This will be achieved by establishing the ''nonvanishing property'' in the spirit of Theorem~3.1 in \cite{EW} or
Theorem~3.1 in \cite{Ev}.
  
\begin{lem} \label{nonvanishing}
  Assume (A1),(A2) and let $(u_n) \subset L^{p^\prime}(\R^N) $ be
  a bounded sequence such that 
  $$
  	\limsup_{n\rightarrow \infty} \Big|\int_{\R^N} u_n \bold{R} u_n \,dx\Big|>0.
  $$ 
  Then there exist $R,\zeta>0$ and a sequence $(x_n) \subset \R^N$ such that we have 
  \begin{equation*}
	\int_{B_R(x_n)} |u_n|^{p^\prime}\,dx \geq \zeta \quad\text{for infinitely many }n\in\N.
  \end{equation*}
\end{lem}
\begin{proof}
We use the same notation as in the proof of Theorem~\ref{thminvopbound}. We recall from
\eqref{gutiee1},\eqref{estG2} and \eqref{estG2a_2=0}, that there exists a $C>0$ such that for all $x\in \R^N$ the fundamental solution
$G=G_1+G_2$ satisfies
\begin{equation}\label{estG2bis1}
  |G_1 (x)|\leq C (1+|x|)^{\frac{1-N}{2}},
\end{equation}
\begin{equation} 
\label{estG2bis2}
 |G_2 (x)|\leq \begin{cases} C \min \{|x|^{4-N}, |x|^{-N}\}&, \text{if }N> 4,\\ 
  C \min \{1+\log|x|, |x|^{-N}\}&,\text{if } N= 4,\\
 C \min \{1, |x|^{-N}\}&,\text{if } N= \{2,3\},\\ \end{cases} \qquad\text{when }\;a_2\neq 0,
\end{equation}
and
\begin{equation}
\label{estG2a_2=0bis}
|G_2 (x)|\leq \begin{cases} C \min \{|x|^{4-N},  |x|^{-N-1}\} &,\text{if}\ N> 4,\\ 
  C \min \{1+ |\log|x||, |x|^{-N-1}\} &,\text{if}\ N= 4,\\
C \min \{1,|x|^{-N-1} \}&, \text{if}\ N=3, \\
\end{cases}
  \qquad\text{when }\;a_2=0.
\end{equation}
 For a sequence $(u_n)$ as required we assume for contradiction that
\begin{equation} \label{nonvae1}
  \lim_{n\rightarrow \infty} \left(\sup_{y\in \R^N}\int_{B_\rho (y)} |u_n|^{p^\prime}\,dx  \right)=0
  \quad \text{for all }\rho>0. 
\end{equation}
From this we will deduce 
\begin{align} \label{nonvaeclaim}
  \int_{\R^N} u_n (G_1 \ast u_n )\,dx\to 0 \;\;\text{and}\;\;\int_{\R^N} u_n (G_2 \ast u_n )\,dx \rightarrow 0  
  \quad\text{as } n\rightarrow \infty,
\end{align}
leading to a contradiction to our assumption.

\medskip

Both claims are proved almost identically as in \cite{Ev,EW} so that we only provide the main steps. 
We first prove the second assertion. For $R>1$ define 
$D_R:= \R^N \backslash A_R$ for the annulus $A_R:=\{x\in \R^N :\ 1/R \leq |x|\leq R \}$.  
Thanks to \eqref{estG2bis2} and \eqref{estG2a_2=0bis} we have $\|G_2 \|_{L^{p/2}(D_R)}\to 0$ as $R\to\infty$ since
$(N-4)p<2N$ and $(N+1)p>2N$. So Young's convolution inequality implies
\begin{equation} \label{nonvae2}
  \sup_{n\in \N} \Big| \int_{\R^N}u_n \left[ (1_{D_R} G_2)\ast u_n \right] \,dx\Big|\leq
 \|G_2\|_{L^{p/2}(D_R)} \sup_{n\in \N}\|u_n\|_{L^{p^\prime}(\R^N)}^2 \rightarrow 0
  \text{ as }R\to \infty.
\end{equation}
On the other hand, in the case $N\geq 2,N\neq 4$ we may use the estimates from the bottom of p.706
in~\cite{EW} with $R^{N-2}$ replaced by $R^{N-4}$ and in the case $N= 4$ the estimates from p.11 in~\cite{Ev}
with $R^{4/p}(1+|\log(R)|)$ replaced by $R^{8/p}(1+|\log(R)|)$ to find
\begin{align}\label{nonvae3}
  \int_{\R^N} u_n  \left[(1_{A_R} G_2)\ast u_n \right] \,dx \to 0 \quad\text{as }n\to\infty\quad
  \text{for all }R>0.
\end{align}
Combining \eqref{nonvae2} and \eqref{nonvae3} we get 
$$
  \int_{\R^N} u_n (G_2 \ast u_n )\,dx \to 0 \quad\text{as }n\to\infty.
$$

\medskip

Next, we turn to the second claim. To this end we set $M_R:= \R^N \backslash B_R$. As in
Proposition~3.3 in \cite{EW} or Claim 2 on p.11 in \cite{Ev} one proves the inequality
\begin{equation*} 
  \|[1_{M_R} G_1]\ast f \|_{L^p (\R^N)} \leq C R^{-(N-1)/2 + (N+1)/p }\|f\|_{L^{p^\prime}(\R^N )}
\end{equation*}
whenever $f\in \mathcal{S} (\R^N)$ satisfies $\supp(\hat f) \subset \{\xi\in\R^N : ||\xi|- \sqrt a_1|\leq
\sqrt a_1/2 \}$.
Notice that $G_1$ satisfies, qualitatively, the same bounds as the function $\Phi_1$ in \cite{EW},
see~\eqref{estG2bis1} and the estimates~(26),(7) in~\cite{EW},\cite{Ev}, respectively. The proof of Lemma~3.4
in \cite{EW} and Claim~3 on p.12 in \cite{Ev} transfers literally to our situation proving 
$$
  \int_{\R^N} u_n (G_1 \ast u_n )\,dx \to 0 \quad\text{as }n\to\infty,
$$
which finishes to proof.
\end{proof}

With these preparations we can finally prove the existence of a nontrivial solution for~\eqref{4nls}.

\begin{thm} \label{thmex}
  Assume (A1),(A2). Then there exists a nontrivial critical point $v\in L^{p^\prime}(\R^N)$ of $J$ at the mountain pass level $c$ defined in
  ~\eqref{eq:MP_level}.
\end{thm}
\begin{proof}
  Let $(v_n) \subset L^{p^\prime}(\R^N)$ be a bounded Palais-Smale sequence for $J$ given by
  Lemma~\ref{PSbounded}~(iii).
  Then we have 
  $$
    \lim_{n\rightarrow \infty} \int_{\R^N}  \Gamma^{1/p}  v_n \bold{R} (  \Gamma^{1/p}  v_n) \,dx
    = \frac{2 p^\prime}{2-p^\prime} \lim_{n\rightarrow \infty} \left( J(v_n) - \frac{J^\prime (v_n) v_n}{p^\prime}\right)
    = \frac{2p^\prime }{2-p^\prime} c>0.
  $$ 
  Hence, Lemma~\ref{nonvanishing} implies that there exist $R,\zeta>0$ and a sequence $(x_n) \subset \R^N$
  such that, up to a subsequence and reindexing $v_n$, we have
  \begin{equation} \label{mainthme1}
	\int_{B_R (x_n)}|v_n|^{p^\prime} \,dx \geq \zeta\quad \text{for all } n\in\N.
  \end{equation}
  Then the sequence $(w_n)\subset L^{p^\prime}(\R^N)$ given by 
  $w_n(x):= v_n (x+x_n)$ is bounded with $J(w_n)\to c$ and $\|J^\prime (w_n)\|=\|J^\prime (v_n)\|\to 0$ as
  $n\to\infty$. For  $\phi\in L^{p^\prime}(\R^N)$ with $\supp(\phi)\subset
  \overline{B_R}$ and all $n,m\in\N$ we get
  \begin{align*}
    &\left| \int_{\R^N} (|w_n|^{p^\prime -2} w_n - |w_m|^{p^\prime -2}w_m ) \phi \,dx  \right| \\
   &= \left|J^\prime (w_n) \phi - J^\prime (w_m) \phi + \int_{B_R}  \bold{R}( \Gamma^{1/p} (w_n - w_m ))   \Gamma^{1/p} \phi\,dx  \right|\\
  &\leq (\|J^\prime (w_n)\| + \|J^\prime (w_m)\|) \|\phi \|_{L^{p^\prime}(\R^N )} 
   + \|  \Gamma^{1/p} 1_{B_R} \bold{R} ( \Gamma^{1/p} (w_n - w_m))\|_{L^p (\R^N)} \|\phi \|_{L^{p^\prime}(\R^N )}.
  \end{align*}
  Using  $\|J^\prime (w_n)\|=\|J^\prime (v_n)\|\to 0$ as well as the compactess of $1_{B_R}\bold R$ (see
  Lemma~\ref{lemcompact}), we obtain that a subsequence of $(|w_n|^{p^\prime -2}w_n)$ is a
  Cauchy sequence in $L^p(B_R)$. So there exists $w\in L^{p^\prime}(B_R)$ such that $|w_n|^{p^\prime -2}w_n
  \rightarrow |w|^{p^\prime -2}w$ strongly in $L^p(B_R)$. Moreover, we deduce from
  $w_n(x)=v_n(x+x_n)$ and~\eqref{mainthme1} that 
  $$
    \int_{B_R} |w|^{p^\prime} \,dx >0,
  $$ 
  which implies $w\neq 0$. Finally we observe that $w$ is a critical point of $J$ because we have for
  all $\phi \in C_0^\infty (\R^N)$ the identity
  \begin{align*}
    J^\prime (w) \phi
    &= \left(\int_{\R^N} |w|^{p^\prime -2} w \phi\,dx - \int_{\R^N}\bold{R}( \Gamma^{1/p} w)  \Gamma^{1/p}  \phi \,dx \right) \\
    &= \lim_{n\to \infty} \left(\int_{\R^N} |w_n|^{p^\prime -2} w_n \phi\,dx - \int_{\R^N} 
    \bold{R}(  \Gamma^{1/p}  w_n)   \Gamma^{1/p} \phi\,dx \right) \\
    &= \lim_{n\rightarrow \infty} J^\prime (w_n) \phi \\ 
    &=0,
  \end{align*} 
  where we used that $\bold{R}$ is a bounded linear operator and that $|w_n|^{p^\prime -2}w_n \rightarrow
  |w|^{p^\prime -2}w$ in $L^p_{loc}(\R^N)$. Therefore, $w\in L^p (\R^N)$ is a nontrivial critical point of
  $J$. 
\end{proof}

\section{Qualitative properties of solutions}\label{sec:qualitative}

In this section we investigate the regularity and the asymptotic behavior of the critical point that we
obtained in Theorem~\ref{thmex}. First we consider the local and global regularity of critical points of $J$
and thus of the solution obtained above. In particular we will see that, not surprisingly, this critical
point is a strong solution of~\eqref{4nls}. Then we investigate its behaviour at infinity in more detail by
establishing its pointwise decay as well as its asymptotics at infinity, also known as the farfield expansion.

\subsection{Regularity of solutions}

  We start by showing a local regularity result for distributional solutions of the linear problem associated
  with~\eqref{4nls}. We refer to \cite{Man_reg} and \cite{BaoZh_homogeneous,BaoZh_nonhomogeneous} for other
  results in this direction.

\begin{prop} \label{propreg}
  Assume (A1). Let $f\in L^q_{loc}(\R^N)$ for $q\in (1,\infty)$ and let 
  $u\in L^q_{loc}(\R^N)$ be a distributional solution of 
  $
    \Delta^2 u-\beta \Delta u + \alpha u =  f \text{ in } \R^N.
  $
  Then $u\in W^{4,q}_{loc}(\R^N)$ is a strong solution and for all $r>0$  there exists a
  constant $C>0$ depending on $r,\ p$ and $N$ such that for all $x_0 \in \R^N$ 
  \begin{equation}\label{eq:CZ_estimates}
      \|u \|_{W^{4,q} (B_r(x_0))}\leq C
      (\|u\|_{L^{q}(B_{2r}(x_0))}+\|f\|_{L^{q}(B_{2r}(x_0))}).
  \end{equation}
\end{prop}
\begin{proof}
   The proof follows the lines of the proof of Proposition A.1 in~\cite{EW}. We use a mollifier $\rho\in
   C_0^\infty(\R^N)$ and set $u_\eps:=u\ast \rho_\eps,f_\eps:=f\ast \rho_\eps$ for
   $\rho_\eps:=\eps^{-N}\rho(\eps^{-1}\cdot)$. Then the equation 
   $$
     \Delta^2 u_\eps -\beta \Delta u_\eps+\alpha u_\eps 
     =  f_\eps \quad \text{ in } \R^N
   $$
   holds in the classical sense.  
   Applying the interior $L^p$-estimates for higher order elliptic problems from Theorem 14.1' in
   \cite{ADN_estimates_I} we get for sufficiently small $\eps>0$ and for all $r>0,x_0\in\R^N$ 
    \begin{align*}
     \|u_\eps\|_{W^{4,q}(B_r(x_0))}  
     &\leq C (\|u_\eps\|_{L^q(B_{3r/2}(x_0))}+\|f_\eps\|_{L^q(B_{3r/2}(x_0))})  \\ 
     &\leq C (\|u\|_{L^q(B_{2r}(x_0))}+\|f \|_{L^q(B_{2r}(x_0))}).
   \end{align*}
   Since $u_\eps-u_\delta$ solves the corresponding homogeneous Dirichlet problem and $u_\eps\to u$
   in $L^q_{loc}(\R^N)$ as $\eps\to 0$, we deduce that $(u_\eps)$ is a Cauchy sequence
   in $W^{4,q}(B_r(x_0))$ as $\eps\to 0$ for any $r>0,x_0\in\R^N$, hence $u$ lies in $W^{4,q}_{loc}(\R^N)$
   and satisfies the estimates~\eqref{eq:CZ_estimates}.
\end{proof}

 We go on with proving global regularity results for solutions of the linear problem.

  \begin{prop} \label{prop:global_reg_linear}
    Assume (A1). Let $f\in L^{p^\prime}(\R^N)\cap L^q(\R^N)$ for $p\geq \frac{2(N+1)}{N-1},q\in
    (1,\infty)$ and $u=\bold{R} f \in L^q(\R^N)$. Then $u\in W^{4,q}(\R^N)$ is a strong
    solution of $
       \Delta^2 u-\beta \Delta u+\alpha u =f \text{ in } \R^N 
     $ 
     and there is a $C>0$ such that 
     $$
       \|u\|_{W^{4,q}(\R^N)} \leq C(\|u\|_{L^q(\R^N)}+\|f\|_{L^q(\R^N)}).
     $$
  \end{prop}
  \begin{proof} 
    We first discuss the case $a_1>0>a_2$. With the notation from the previous
    proposition we have
    \begin{equation} \label{eq:ueps_veps_system}
      -\Delta u_\eps - a_1  u_\eps = v_\eps\quad\text{in }\R^N,\qquad
      -\Delta v_\eps - a_2  v_\eps = \sqrt{\beta^2-4\alpha} f_\eps\quad\text{in }\R^N
    \end{equation}
    for some function $v_\eps\in W^{2,q}(\R^N)$. Thanks to $a_2<0$, the second equation and global
    $L^p$-estimates (see for instance Theorem C.1.3.(iii) in \cite{LorBer_analytical}) 
    imply that $(v_\eps)$ is a Cauchy sequence in $W^{2,q}(\R^N)$ and satisfies
    $$
      \|v_\eps\|_{W^{2,q}(\R^N)}\leq C \|f_\eps\|_{L^q(\R^N)} \leq C \|f\|_{L^q(\R^N)}.
    $$
    From the first equation for $u_\eps$ we deduce 
    $$
      -\Delta (\Delta u_\eps) + \Delta u_\eps = (a_1+1)\Delta u_\eps + \Delta v_\eps,
    $$ 
    so that the same estimates as above together with interpolation estimates imply  
    \begin{align*}
      \|u_\eps\|_{W^{4,q}(\R^N)}
      &\leq C\|\Delta u_\eps\|_{W^{2,q}(\R^N)} \\
      &\leq C(\|(a_\eps+1)\Delta u_\eps + \Delta v_\eps\|_{L^q(\R^N)}) \\
      &\leq C (\|u_\eps\|_{W^{2,q}(\R^N)} + \|v_\eps\|_{W^{2,q}(\R^N)}) \\
      &\leq \frac{1}{2} \|u_\eps\|_{W^{4,q}(\R^N)} + C(\|u_\eps\|_{L^q(\R^N)} + \|f_\eps\|_{L^q(\R^N)})
      \\
      &\leq \frac{1}{2} \|u_\eps\|_{W^{4,q}(\R^N)} + C(\|u\|_{L^q(\R^N)}+\|f\|_{L^q(\R^N)}),
    \end{align*}
    which proves the boundedness of $(u_\eps)$ in $W^{4,q}(\R^N)$. Performing the corresponding estimates for
    $u_\eps-u_\delta$, which solves \eqref{eq:ueps_veps_system} with $f_\eps$ replaced by 0, we find that 
    $(u_\eps)$ is a Cauchy sequence as $\eps\to 0$ in $W^{4,q}(\R^N)$ converging to
    $u\in W^{4,q}(\R^N)$ satisfying also the above estimate.
    
    \medskip
    
    In the other case $a_1>a_2\geq 0$ we rewrite the equation as
    $
      \Delta^2 u-\beta \Delta u - u = \tilde f  \text{ in } \R^N 
    $
    where $\tilde f:= f-(1+\alpha)u$. Now the symbol of the differential operator on the left hand side has
    one positive and one negative zero and $\|\tilde f\|_{L^q(\R^N)} \leq C(\|f\|_{L^q(\R^N)} + \|u\|_{L^q(\R^N)})$, 
    so that our considerations from above yield the result.
  \end{proof}
  
  With these preliminary results we deduce the regularity of the critical point constructed in
  Theorem~\ref{mainthm}.
  
  \begin{thm}\label{thm_regularity}
     Assume (A1),(A2) and let $u\in L^p(\R^N)$ be a solution to $u=\bold{R}(\Gamma |u|^{p-2}u)$. Then $u\in
     W^{4,q}(\R^N)\cap C^{3,\alpha}(\R^N)$ for all $q\in [p,\infty),\alpha\in [0,1)$ and $u$ 
     is a strong solution of
     $$
       \Delta^2 u-\beta \Delta u+\alpha u = \Gamma|u|^{p-2}u \quad \text{in } \R^N.
     $$ 
  \end{thm}
  \begin{proof}
    It suffices to prove $u\in L^\infty(\R^N)$. Indeed, having shown this, we may apply
    Proposition~\ref{prop:global_reg_linear} to $f:=\Gamma |u|^{p-2}u \in L^{p^\prime}(\R^N)\cap
    L^q(\R^N)$ for all $q\in [p^\prime,\infty)$. Proposition~\ref{prop:global_reg_linear} yields $u\in W^{4,q}(\R^N)$
    for all $q\in [p,\infty)$ and thus, by Morrey's imbedding Theorem, $u\in C^{3,\alpha}(\R^N)$ for all
    $\alpha\in [0,1)$. In order to prove the boundedness of $u$ we iterate the local estimates from
    Proposition~\ref{propreg}.  From Sobolev's imbedding theorem and Proposition~\ref{propreg} we get for all
    $q\in [1,\infty]$ and all $s\geq \frac{Nq}{N+4q},s>1$  
    \begin{align*}
      \|u\|_{L^q(B_r(x_0))}
      &\leq C \|u\|_{W^{4,s}(B_r(x_0))} \\
      &\leq  C(\|u\|_{L^s(B_{2r}(x_0))}+\|\Gamma
      |u|^{p-2}u\|_{L^s(B_{2r}(x_0))}) \\
       &\leq  C(1+\|u\|_{L^{s(p-1)}(B_{2r}(x_0))}^{p-1})
    \end{align*}
    for some positive $C$ dependent of $r,N,s,q,\|\Gamma\|_{L^\infty(\R^N)}$ but not on $x_0$. Hence, for 
    $q_0:=\infty$ and $q_{n+1}:=\max\{\frac{Nq_n}{N+4q_n}(p-1),p\}$ we get
    $$
      \|u\|_{L^{q_n}(B_{2^nr}(x_0))} \leq C_n(1+\|u\|_{L^{q_{n+1}}(B_{2^{n+1}r}(x_0))})^{p-1} \quad\text{for
      }n\in\N_0.
    $$
    Since $q_n\geq p>\frac{N(p-2)}{4}$ the sequence $(q_n)$ decreases until it reaches the
    value $p$ after finitely many steps. This implies  
    \begin{align*}
      \|u\|_{L^\infty(B_r(x_0))}
      \leq P(\|u\|_{L^p(\R^N)})
    \end{align*} 
    for some polynomial $P$ with positive coefficients and the proof is finished, since $P$ does not depend
    on $x_0$.
  
  \end{proof}

\subsection{Decay and farfield expansion}

In this section we establish the pointwise decay and the farfield expansion of the solution obtained in
Theorem~\ref{thmex}. The proof of this result follows again the lines of earlier results due to Ev\'{e}quoz
and Weth.  
 
\begin{thm}\label{thm:decay}
  Assume (A1),(A2) and $u=\bold{R}(f)$ where $f:=\Gamma |u|^{p-2}u\in L^{p^\prime}(\R^N)$.
  Then  
  \begin{equation} \label{eqfarfield}
    \lim_{R\rightarrow \infty} \frac{1}{R}\int_{B_R} | u(x) -\Real(U_f )(x)|^2 \,dx =0. 
  \end{equation}
  Moreover, if $N\in\{2,3\}$ or $N\geq 4,p>\frac{3N-1}{N-1}$, then we have
  \begin{itemize}
    \item[(i)] There is a $C>0$ such that $|u(x)|\leq C (1+|x|)^{\frac{1-N}{2}}$ for all $x\in\R^N$.
    \item[(ii)] $u(x) =\Real(U_f)(x) + o(|x|^{\frac{1-N}{2}})$ as $|x|\to\infty$. 
  \end{itemize}
\end{thm}
\begin{proof}
  From Theorem~\ref{thminvopbound} we get that $u\in L^r(\R^N)$ implies $f\in L^{r/(p-1)}(\R^N)$ and
  thus $u\in L^{\tilde r}(\R^N)$ whenever 
  $$
    \tilde r:= \frac{r(N+1)}{-2r+(p-1)(N+1)}\quad\text{and}\quad
    \frac{2N(N+1)}{N^2+4N-1}(p-1)=:r_* <r\leq p.
  $$
  So we put $r_0:= p$ and define $r_n$ inductively via
  $$
    r_{n+1} := \frac{1}{2} \Big( r_n + \max\big\{ \frac{r_n(N+1)}{-2r_n+(p-1)(N+1)}, r_* \big\}\Big). 
  $$
  Then $r\leq p<\frac{N+1}{2}(p-2)$ allows to prove inductively $r_*<r_{n+1}<r_n\leq r$ for all $n\in\N_0$
  so that the sequence $(r_n)$ strictly decreases
  to $r_*$, which gives $u\in L^\infty(\R^N)\cap L^r(\R^N)$ for all $r>\frac{2N}{N-1}$. (The boundedness was proved in
  Theorem~\ref{thm_regularity}.) In particular, we have $f=\Gamma |u|^{p-2}u \in L^{r'}(\R^N)$ for
  $r=\frac{2(N+1)}{N-1}$ as well as the representation formula 
  \begin{align} \label{eq:representation_formula_for_u} 
    u(x) 
    = (\Real G\ast f)(x) 
    = \frac{1}{\sqrt{\beta^2-4\alpha}} \Big( ( \Real g_{a_1}\ast f)(x) - (\Real g_{a_2}\ast f)(x)\Big), 
  \end{align}
  see~\eqref{eq:def_G}.  Our strategy is to prove first~\eqref{eqfarfield} for $N\geq 4$. Then we show 
  ~(i),(ii) for $N\in\{2,3\}$ or $N\geq 4,p>\frac{3N-1}{N-1}$. Since~(i),(ii)
  implies~\eqref{eqfarfield}, this will prove the assertion.
  
  \medskip
  
  In the case $N\geq 4$ and $a_1>a_2>0$ we may directly deduce \eqref{eqfarfield} from Proposition~2.7
  in~\cite{EW}. Indeed,  the
  formula~\eqref{eq:representation_formula_for_u} and~\eqref{eq:EWresolvent_vs_ga1} yields 
  $$
    u(x)  
    =  \frac{1}{\sqrt{\beta^2-4\alpha}} \Real\Big( \frac{1}{a_1} \mathcal{R}_1\Big(f(\frac{\cdot}{\sqrt
    a_1})\Big)(\sqrt a_1 x)  -  \frac{1}{a_2} \mathcal{R}_1\Big(f(\frac{\cdot}{\sqrt
    a_2})\Big)(\sqrt a_2 x)\Big)
  $$
  where $\mathcal{R}_1$ is the resolvent for the Helmholtz operator $-\Delta  -1$ studied in~\cite{Ev,EW}.
  Using $f\in L^{r'}(\R^N)$ for $r=\frac{2(N+1)}{N-1}$ (see above) and Proposition~2.7 in~\cite{EW} yields
  \eqref{eqfarfield} because of
  \begin{align*}  
    & \frac{1}{\sqrt{\beta^2-4\alpha}} \frac{1}{a} \sqrt{\frac{\pi}{2}} \frac{e^{i(\sqrt{a}
    |x|-\frac{N-3}{4}\pi)}}{|\sqrt{a} x|^{\frac{N-1}{2}}}   \widehat{f\big(\frac{\cdot}{\sqrt
    a}\big)}\Big(\frac{\sqrt a x}{|\sqrt a x|}\Big)  \\
    &= \frac{1}{\sqrt{\beta^2-4\alpha}} \frac{1}{a}\sqrt{\frac{\pi}{2}} \frac{e^{i(\sqrt{a}
    |x|-\frac{N-3}{4}\pi)}}{|\sqrt{a} x|^{\frac{N-1}{2}}}   \sqrt{a}^N \hat{f}(\sqrt a \frac{x}{|x|}) \\
    &= \frac{a^{\frac{N-3}{4}}}{\sqrt{\beta^2-4\alpha}} \sqrt{\frac{\pi}{2}} \frac{e^{i(\sqrt{a}
    |x|-\frac{N-3}{4}\pi)}}{|x|^{\frac{N-1}{2}}} \hat{f}\big(\sqrt{a} \frac{x}{|x|}\big)
  \end{align*}
  for $a\in\{a_1,a_2\}$, see the definition of $U_f$ in front of Theorem~\ref{mainthm}. 
  
  \medskip
  
  In the case $N\geq 4$ and $a_1>0\geq a_2$ we show that result from the above-mentioned Proposition remains true
  after slight modification. To this end we first consider $\tilde f\in C_0^\infty(\R^N)$. Exactly the same
  proof as Proposition~2.6 in~\cite{EW} yields $(\mathfrak R \tilde f)(x) = U_{\tilde f}(x) +
  o(|x|^{(1-N)/2})$. Indeed, the proof of this proposition only exploits the  formula
  $$
    (\mathcal{R}_1\tilde f)(x) = \gamma_N \int_{B_R}
    \frac{e^{i|x-y|}}{|x-y|^{(N-1)/2}}(1+\delta(|x-y|))\tilde f(y)\,dy 
  $$
  for $|x|\geq 2R$ and $R$ is chosen so large that $\supp(\tilde f)\subset B_R$ holds. Here, the function
  $\delta$ satisfies $\sup_{r\geq 1}r|\delta(r)|<\infty$ where $\gamma_N =
  \frac{1}{2}(2\pi)^{(1-N)/2}e^{-i(N-3)\pi/4}$, see the bottom of~p.697 in~\cite{EW}. In view of 
  $$
    (\mathfrak{R}\tilde f)(x) = \frac{\gamma_N}{\sqrt{\beta^2-4\alpha}} 
    \int_{B_R} \frac{e^{i|x-y|}}{|x-y|^{(N-1)/2}}(1+\tilde \delta(|x-y|))\tilde f(y)\,dy 
  $$
  for $|x|\geq 2R$ large enough and $\sup_{r\geq 1}\sqrt{r}|\tilde
  \delta(r)|<\infty$ we therefore get $(\mathfrak R \tilde f)(x) = U_{\tilde f}(x) + o(|x|^{(1-N)/2})$
  whenever $\tilde f\in C_0^\infty(\R^N)$.
  Notice that this property of $\tilde\delta$ follows from \eqref{asymgreen}, i.e., from the fact that
  the Green's function $g_{a_2}$ decays exponentially if $a_2<0$ or like $|x|^{2-N}$ if $a_2=0$ and hence
  faster than $g_{a_1}$ at infinity. With this result we deduce~\eqref{eqfarfield} by approximation of $f$ in $L^{r^\prime}(\R^N)$ by test
  functions exactly as in the proof of Proposition~2.7~\cite{EW}.  
  
  \medskip

  Now let us assume $N=3$ or $N\geq 4,p>\frac{3N-1}{N-1}$. Then we have $f=Vu$ and $u= \Real G\ast
  Vu$ for $V=\Gamma|u|^{p-2}\in L^q(\R^N)\cap L^s(\R^N)$ with $s=\infty>\frac{N}{2}$ and $q<\frac{2N}{N+1}$
  because of $\frac{2N}{N+1}>\frac{2N}{(N-1)(p-2)}$. Furthermore, we have $Vu \in L^1(\R^N)\cap L^\infty(\R^N)$ due to
  $\frac{2N}{(N-1)(p-1)}<1$. Exploiting the estimate $|G(z)|\leq C\max\{|z|^{2-N},|z|^{\frac{1-N}{2}}\}$ 
  claim (i) follows from Lemma~2.9 in~\cite{EW}. In particular, $p>\frac{3N-1}{N-1}$ implies  
  $|f(x)|=\Gamma(x)|u(x)|^{p-1}\leq C(1+|x|)^{-N-\delta}$ for some $\delta>0$ and all $x\in\R^N$, so that
  \eqref{eq:representation_formula_for_u} together with Proposition~2.8 in~\cite{EW} ($a_2>0$) or
  the proof of Claim 2 in Proposition~6.3 in \cite{BCGJ2} ($a_2=0$), respectively, gives
  \begin{align*}
    u(x) 
    = ( \Real G\ast f)(x) 
    = \Real(U_f)(x) + o(|x|^{\frac{1-N}{2}}).
  \end{align*}
  In the case $N=2$ the corresponding result follows from Proposition~2.2 in~\cite{Ev}. 
  This finishes the proof.  
\end{proof}

\medskip

Finally let us add that the solution $u=\bold{R}(\Gamma|u|^{p-2}u)$ described in Theorem~\ref{thm:decay} is
the real part of a complex-valued solution 
$$
   \tilde u:=\mathfrak R(\Gamma|u|^{p-2}u) = G\ast (\Gamma|u|^{p-2}u). 
$$ 
As above, one shows that in the case $a_1>0>a_2$ this function satisfies Sommerfeld's outgoing radiation
condition in the following integral sense
\begin{equation}
\label{some2}
  \lim_{R\rightarrow \infty}
  \frac{1}{R}\int_{B_R} \Big|\nabla \tilde u (x) - i \sqrt{a_1} \tilde{u}(x) \frac{x}{|x|} \Big|^2 \,dx =0
  \quad (a_1>0>a_2),
\end{equation} 
 see equation~(53) in~\cite{EW} for the corresponding result in the Helmholtz case. Notice that 
 in this way the solution $\tilde u$ inherits the radiation condition from the fundamental solution $G$, 
 see~\eqref{sommerfeld}. In the case $a_1>a_2>0$,
 however, defining $\tilde u_j:= g_{a_j}\ast (\Gamma|u|^{p-2}u)$ for $j=1,2$   we get   
\begin{equation}
\label{some3}
  \tilde u=\frac{\tilde u_1-\tilde u_2}{\sqrt{\beta^2-4\alpha}} \quad\text{with } 
  \lim_{R\rightarrow \infty}
  \frac{1}{R}\int_{B_R} \Big|\nabla \tilde u_j (x) - i \sqrt{a_j} \tilde{u_j}(x) \frac{x}{|x|} \Big|^2 \,dx
  =0 \quad (a_1>a_2>0).
\end{equation}

\section*{Acknowledgements}

D. Bonheure \& J.B. Casteras are supported by MIS F.4508.14 (FNRS), PDR T.1110.14F (FNRS); J.B. Casteras is supported by the
Belgian Fonds de la Recherche Scientifique -- FNRS;
D. Bonheure is partially supported by the project ERC Advanced Grant  2013 n. 339958: ``Complex Patterns for
Strongly Interacting Dynamical Systems - COMPAT'' and by ARC AUWB-2012-12/17-ULB1- IAPAS. R. Mandel
gratefully acknowledges financial support by the Deutsche Forschungsgemeinschaft (DFG, German Research
Foundation) through the Collaborative Research Center 1173.

\bibliographystyle{plain}
\bibliography{bibhelmholtz}

\end{document}